\documentclass{article}
\usepackage{amsmath}
\usepackage{amsthm}
\usepackage{verbatim}
\usepackage{graphicx}
\usepackage[margin=1in]{geometry}

\newtheorem{theorem}{Theorem}
\newtheorem{lemma}{Lemma}

\theoremstyle{definition}
\newtheorem*{definition}{Definition}

\newcommand{\set}[1]{\{{#1}\}}

\newcommand{\faceS}[1]{\mathcal{S}_{#1}}
\newcommand{\periS}[1]{peri(\mathcal{S}_{#1})}
\newcommand{\gridS}[1]{grid(\mathcal{S}_{#1})}

\begin{document}

\title{On the covering of $n + \epsilon$ square with $n^2 + 1$ unit squares for $n \geq 4$}

\markright{Submission}

\author{Sira Sriswasdi\\ 
\scriptsize Research Affairs, Faculty of Medicine, Chulalongkorn University\\ 
\scriptsize sira.sr@chula.ac.th}

\maketitle

\begin{abstract}
 In 2006 \cite{soifer2006}, it was conjectured that one cannot fully cover a square of side length $> n$ with $n^2 + O(1)$ unit squares. For small $n \in \set{2, 3}$, in 2009, \cite{janusz2009} proved that it is impossible to fully cover a square of side length $> n$ with exactly $n^2 + 1$ unit squares. Recently in 2023, \cite{baek2023} proved that it is impossible to fully cover an equilateral triangle of side length $> n$ with exactly $n^2 + 1$ unit equilateral triangles whose sides are parallel to it. There were also some progress on a related problem: what is the largest square with side length $S(k)$ that can be fully covered by $k$ unit squares \cite{friedman2006,dosa2026}. However, there have been no improvement on the original conjecture. In this work, new tools have been developed that led to the proof of this conjecture for $n = 4$, with potential applications to related problems.
\end{abstract}

\noindent

\section{Problem}
In 2006, Alexander Soifer \cite{soifer2006} showed an general strategy for fully covering a square with side length slightly larger than $n$ using $n^2 + o(1)n + O(1)$ unit squares and conjectured that a full covering cannot be achieved with only $n^2 + O(1)$ unit squares. For small $n$, this conjecture simplifies to whether one can fully cover a square with side length slightly larger than $n$ using $n^2 + 1$ unit squares. In 2009, Janusz Januszewski \cite{janusz2009} proved that it is indeed impossible to do so for $n \in \set{2, 3}$. Since then, there has been no progress on this conjecture.

In parallel, studies have been done on a similar problem, such as the covering of equilateral triangles \cite{baek2023}, or from a different perspective, such as determining the largest value $S(k)$ such that a square with side length $S(k)$ can be fully covered by $k$ unit squares \cite{friedman2006,dosa2026}. Some of these results improved the known bounds for the number of unit squares required to achieve a full covering. For example, Dosa, G. \textit{et al.} \cite{dosa2026} proved that a square with side length slightly larger than $2$ cannot be fully covered by $6$ unit squares. Furthermore, as we will see below, the consideration of the covering of the perimeter that was proposed in Dosa, G. \textit{et al.} \cite{dosa2026} is an important tool in our proof.

\section{Definition}
Here, we define the notations concerning parts of the square and, given its covering, the number of unit squares that were placed in certain configurations which will be reused throughout the proof.

\begin{definition}
We denote by $\faceS{n}$ the square with side length $n + \epsilon$, for arbitrarily small $\epsilon > 0$, and $\periS{n}$ its perimeter. 
\end{definition}

Hence, the length of $\periS{n}$ is $4(n+\epsilon) > 4n$.

\begin{definition}
We denote by $\gridS{n}$ the set of $(n + 1) \times (n + 1)$ evenly-spaced line segments parallel to the sides of $\faceS{n}$ (including the sides of $\faceS{n}$ itself) that divide $\faceS{n}$ into $n^2$ equal cells. Note that our definition of $\gridS{n}$ contain $\periS{n}$.
\end{definition}

Hence, the total length of $\gridS{n}$ is $2(n + 1)(n+\epsilon) > 2n(n+1)$. Furthermore, parallel line segments in $\gridS{n}$ are of distance $\frac{n + \epsilon}{n} > 1$ apart.

\begin{definition}
Given a covering of $\faceS{n}$ with unit squares, we call a unit square that intersects with two sides of $\periS{n}$ (including one that covers a corner, as a corner of $\faceS{n}$ is considered to be part of two sides) as a \textit{double-sided perimeter tile} and one that intersects with exactly one side of $\periS{n}$ as a \textit{single-sided perimeter tile}. Lastly, we denote by \textit{interior tiles} the unit squares that lie strictly in the interior of $\faceS{n}$. 
\end{definition}

Because for $n > 1$, a unit square tiles cannot intersect with opposite sides of $\periS{n}$, any \textit{double-sided perimeter tile} must intersect with two adjacent sides of $\faceS{n}$ that form a corner. Thus, we can unambiguously assign each \textit{double-sided perimeter tile} to a corner of $\faceS{n}$.

Also, it is clear that $\periS{n}$ can only be covered by \textit{double-sided perimeter tile} and \textit{single-sided perimeter tile}.

\begin{figure}[h]
    \centering
    \includegraphics[width=0.8\linewidth]{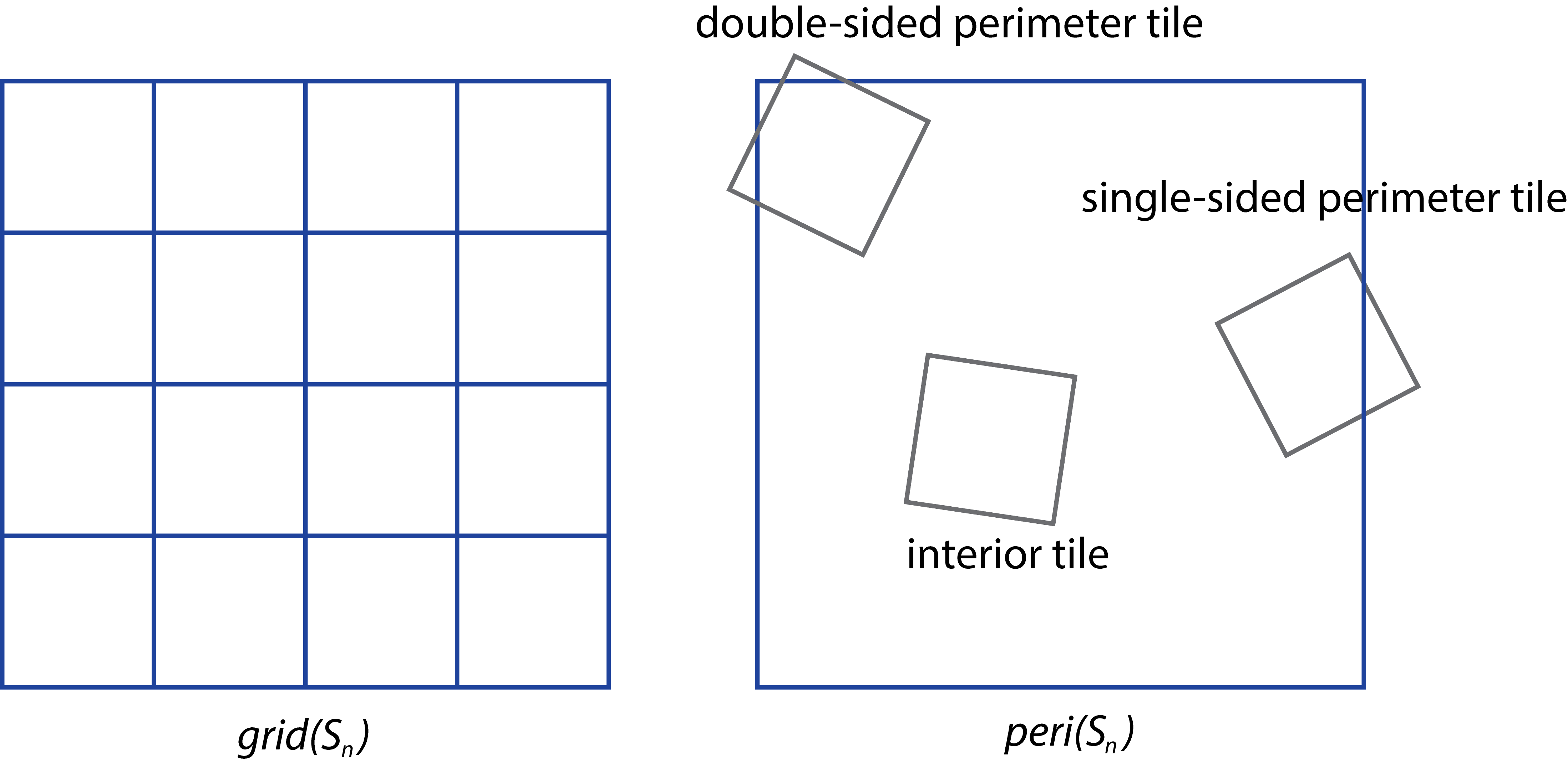}
    \caption{Illustration for $\periS{n}$, $\gridS{n}$, and different types of tiles}
    \label{fig:definition}
\end{figure}

\begin{definition}
Given a covering of $\faceS{n}$ with unit squares, for each corner of $\faceS{n}$, we denote by $n_c$ the number of corners of $\faceS{n}$ that are associated with at least $2$ \textit{double-sided perimeter tiles}. Clearly, $0 \leq n_c \leq 4$.
\end{definition}

\begin{definition}
Given a covering of $\faceS{n}$ with unit squares, we denote by $n_s$ the number of \textit{single-sided perimeter tiles} and $n_d$ the number of \textit{double-sided perimeter tiles}. Clearly, $n_d \geq 4$ because each corner of $\faceS{n}$ must be covered by a different unit square.
\end{definition}

\section{Optimal coverage of $\periS{n}$ and $\gridS{n}$ by unit square tiles}
Here, we develop some useful lemmas that limit how much each unit square tile can contribute to covering of the $\periS{n}$ and $\gridS{n}$, which in turns constrain the number of unit squares needed to fully cover $\periS{n}$ and the interior part of $\gridS{n}$.

The first three lemmas concern the maximum coverage of $\periS{n}$ and $\gridS{n}$ achievable by a \textit{single-sided perimeter tile} or an \textit{interior tile}.

\begin{lemma}\label{lemma:single-sided-cover-peri}
A \textit{single-sided perimeter tile} cannot cover $\periS{n}$ by more than $\sqrt{2}$ unit length.
    \begin{proof}
    This is trivial, as any two points covered by a unit square cannot be more than $\sqrt{2}$ apart. The maximum coverage is achieved by a 45-degree tilted configuration where the unit square's diagonal coincides with a side of $\periS{n}$ (Figure \ref{fig:optimal-cover-summary}).
    \end{proof}
\end{lemma}

\begin{figure}[h]
    \centering
    \includegraphics[width=0.8\linewidth]{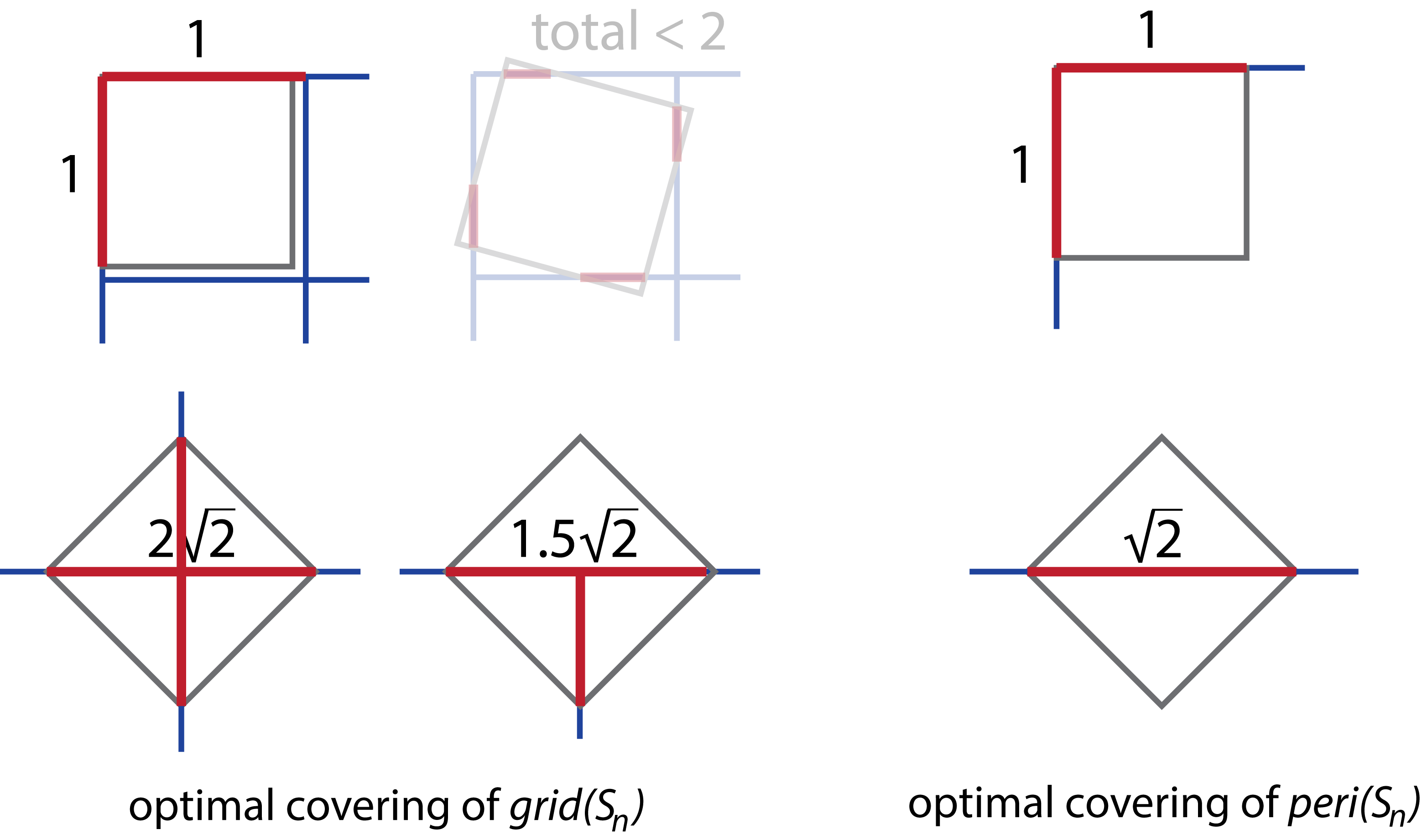}
    \caption{Configurations of a unit square to achieve the maximal coverage of $\periS{n}$ or $\gridS{n}$}
    \label{fig:optimal-cover-summary}
\end{figure}

\begin{lemma}\label{lemma:single-sided-cover-grid}
A \textit{single-sided perimeter tile} cannot cover $\gridS{n}$ by more than $1.5\sqrt{2}$ unit length.
    \begin{proof}
    WLOG, we consider a \textit{single-sided perimeter tile} with vertices $ABCD$ that intersect with the top side of $\gridS{n}$. For this \textit{single-sided perimeter tile} to achieve the maximal covering a $\gridS{n}$, we can follow the strategy outlined in (Figure \ref{fig:double-sided-cover-grid}). 

    The first key idea (Step 1) is that if the lowest vertex of $ABCD$ ($A$ in Figure \ref{fig:double-sided-cover-grid}) does not coincide with a vertical segment of $\gridS{n}$, we can translate $ABCD$ horizontally to maintain its coverage of the top side of $\gridS{n}$ while increasing its coverage of the vertical grid segment.

    The second key idea (Step 2) is that if one of the two adjacent vertices to $A$ ($B$ and $C$ in Figure \ref{fig:double-sided-cover-grid}) is above the top side of $\gridS{n}$, we can translate $ABCD$ vertically downwards until both of them are on or below the top side to maintain its coverage of the top side of $\gridS{n}$ while increasing its coverage of the vertical grid segment.

    Hence, the configuration of $ABCD$ that achieves the optimal coverage of $\gridS{n}$ must have its lowest vertex lying on a vertical segment of $\gridS{n}$ and its second highest vertex lying on the top side of $\gridS{n}$. WLOG, let the second highest vertex be $B$ and let the side $BC$ forms an angle $0 \leq \theta \leq \frac{\pi}{4}$ with the top side of $\gridS{n}$ (Figure \ref{fig:double-sided-cover-grid}). We can calculate the coverage of $\gridS{n}$ by $ABCD$ as a function of $\theta$ as follow:
    \begin{align*}
        d_\theta &= \frac{1}{\cos{\theta}} \text{, } h_\theta = \cos{\theta} \\
        \text{coverage} &= d_\theta + h_\theta = \frac{1}{\cos{\theta}} + \cos{\theta} \leq \sqrt{2} + \frac{1}{\sqrt{2}} = 1.5\sqrt{2}
    \end{align*}
    with the last inequality following from the fact that the function $\frac{1}{\cos{x}} + \cos{x}$ is increasing on $[0, \frac{\pi}{4}]$.

    Therefore, the maximal coverage of $\gridS{n}$ is achieved by placing a 45-degree tilted unit square so that its centroid coincides with the intersection between a side of $\periS{n}$ and an internal line segment of $\gridS{n}$ (Figure \ref{fig:optimal-cover-summary}). This configuration achieves a coverage of exactly $1.5\sqrt{2}$ unit length.

    \begin{figure}[h]
        \centering
        \includegraphics[width=0.8\linewidth]{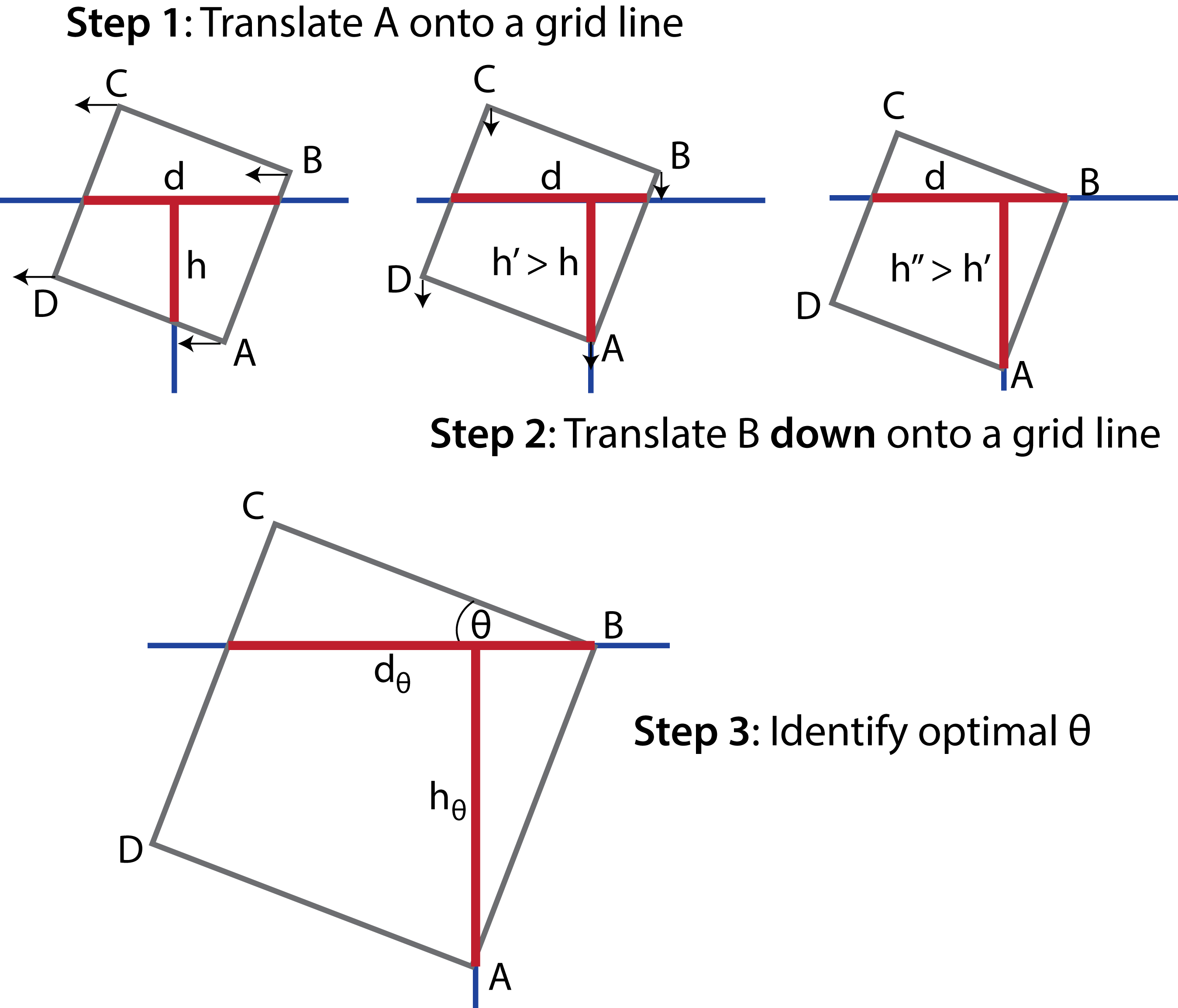}
        \caption{Transformation strategy to achieve the maximal coverage of $\gridS{n}$ by a \textit{single-sided perimeter tile}}
        \label{fig:double-sided-cover-grid}
    \end{figure}
    \end{proof}
\end{lemma}

\begin{lemma}\label{lemma:interior-cover-grid}
An \textit{interior tile} cannot cover $\gridS{n}$ by more than $2\sqrt{2}$ unit length.
    \begin{proof}
    If an \textit{interior tile} intersects with two parallel line segments of $\gridS{n}$ (vertical or horizontal) which are $>1$ unit apart, as shown in \cite{janusz2009} (Figure \ref{fig:janusz-cover}), this \textit{interior tile} cannot cover the $\gridS{n}$ by more than $1$ unit length in that orientation (vertical or horizontal). On the other hand, if an \textit{interior tile} intersects with only one parallel line segment of $\gridS{n}$ (vertical or horizontal), it can achieve a maximum coverage of $\sqrt{2}$ unit length. 
    
    Hence, the maximum coverage of $\gridS{n}$ by an \textit{interior tile} is $2\sqrt{2}$. This is achieved by a 45-degree tilted configuration where the unit square's centroid coincides with an intersection point between a horizontal and a vertical line segment of $\gridS{n}$ (Figure \ref{fig:optimal-cover-summary}).

    \begin{figure}[h]
        \centering
        \includegraphics[width=0.5\linewidth]{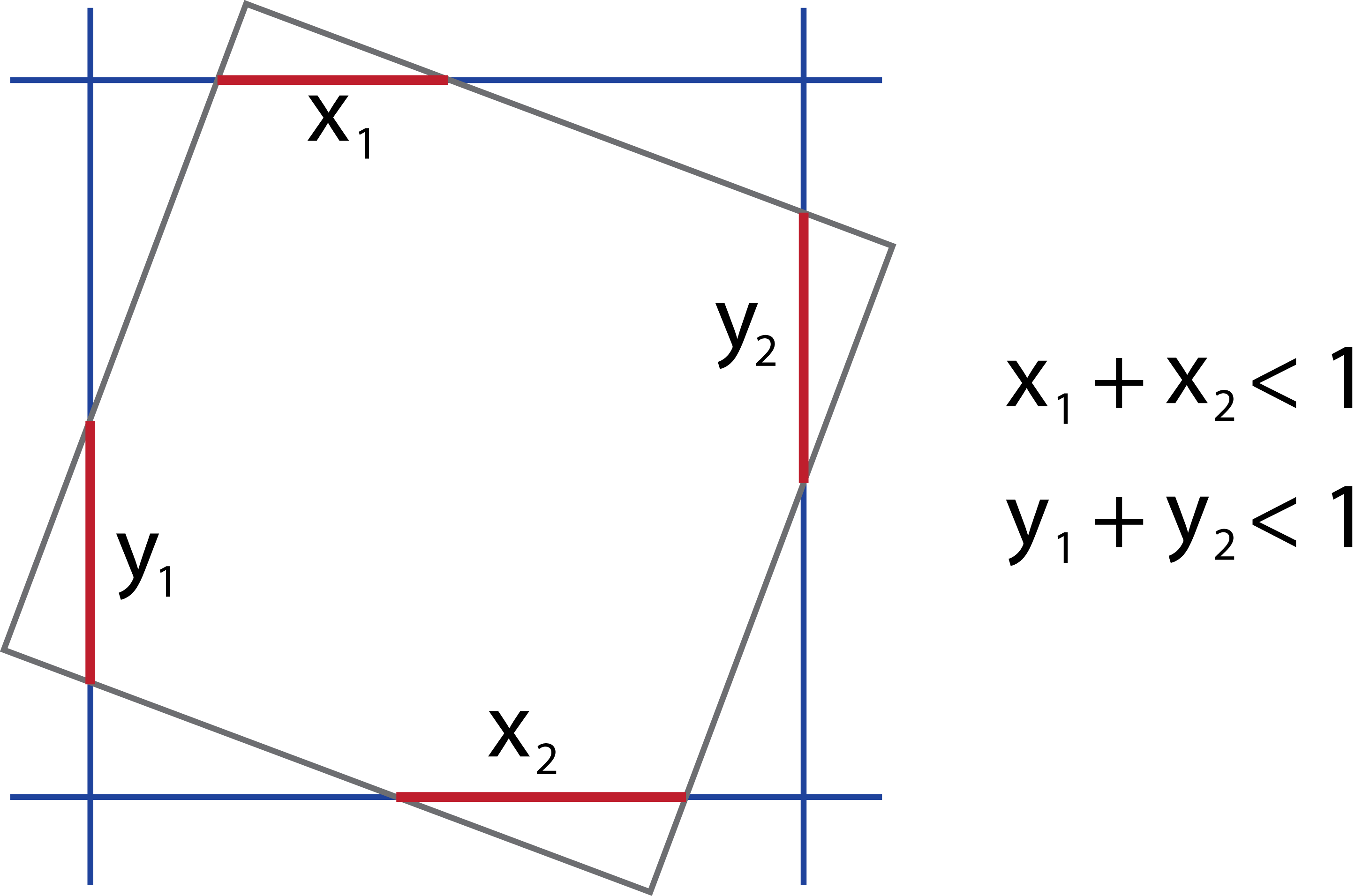}
        \caption{Result from \cite{janusz2009} when a unit square intersects two parallel grid segments of distance $>1$ apart}
        \label{fig:janusz-cover}
    \end{figure}
    \end{proof}
\end{lemma}

The next two lemmas concern the maximum coverage of $\periS{n}$ and $\gridS{n}$ achievable by \textit{double-sided perimeter tiles}.

\begin{lemma}\label{lemma:double-sided-cover}
A \textit{double-sided perimeter tile} cannot cover $\periS{n}$ or $\gridS{n}$ by more than $2$ unit length.
    \begin{proof}
    This followed directly from the first lemma in \cite{janusz2009}. The maximum coverage is achieved by aligning the unit square with the corner of $\faceS{n}$ (Figure \ref{fig:optimal-cover-summary}).
    \end{proof}
\end{lemma}

\begin{lemma}\label{lemma:corner-cover}
All \textit{double-sided perimeter tiles} that are associated with the same corner of $\faceS{n}$ collectively cannot cover $\periS{n}$ by more than $2\sqrt{2}$ unit length.
    \begin{proof}
    Let $l_1$ and $l_2$ be the two sides of $\periS{n}$ that meets at a corner $C$. Let $\mathcal{D}$ be the collection of \textit{double-sided perimeter tiles} that intersect with both $l_1$ and $l_2$. Let $P_1$ and $P_2$ be points on $l_1$ and $l_2$ that are furthest away from the $C$ and are covered by some unit squares in $\mathcal{D}$ (Figure \ref{fig:corner-cover}).

    It is clear that $|CP_1|, |CP_2| \leq \sqrt{2}$ because each unit square in $\mathcal{D}$ intersects with both $l_1$ and $l_2$ and $|CP_1|, |CP_2|$ are the shortest distances from $P_1$ and $P_2$ to $l_2$ and $l_1$, respectively.

    Therefore, all unit squares in $\mathcal{D}$ collectively cover $\periS{n}$ by $|CP_1| + |CP_2| \leq 2\sqrt{2}$. The maximum coverage is achieved by placing two 45-degree tilted unit squares so that one of their corners coincides with $C$ and their diagonals coincide with $l_1$ and $l_2$ (Figure \ref{fig:corner-cover}).

    \begin{figure}[h]
        \centering
        \includegraphics[width=0.8\linewidth]{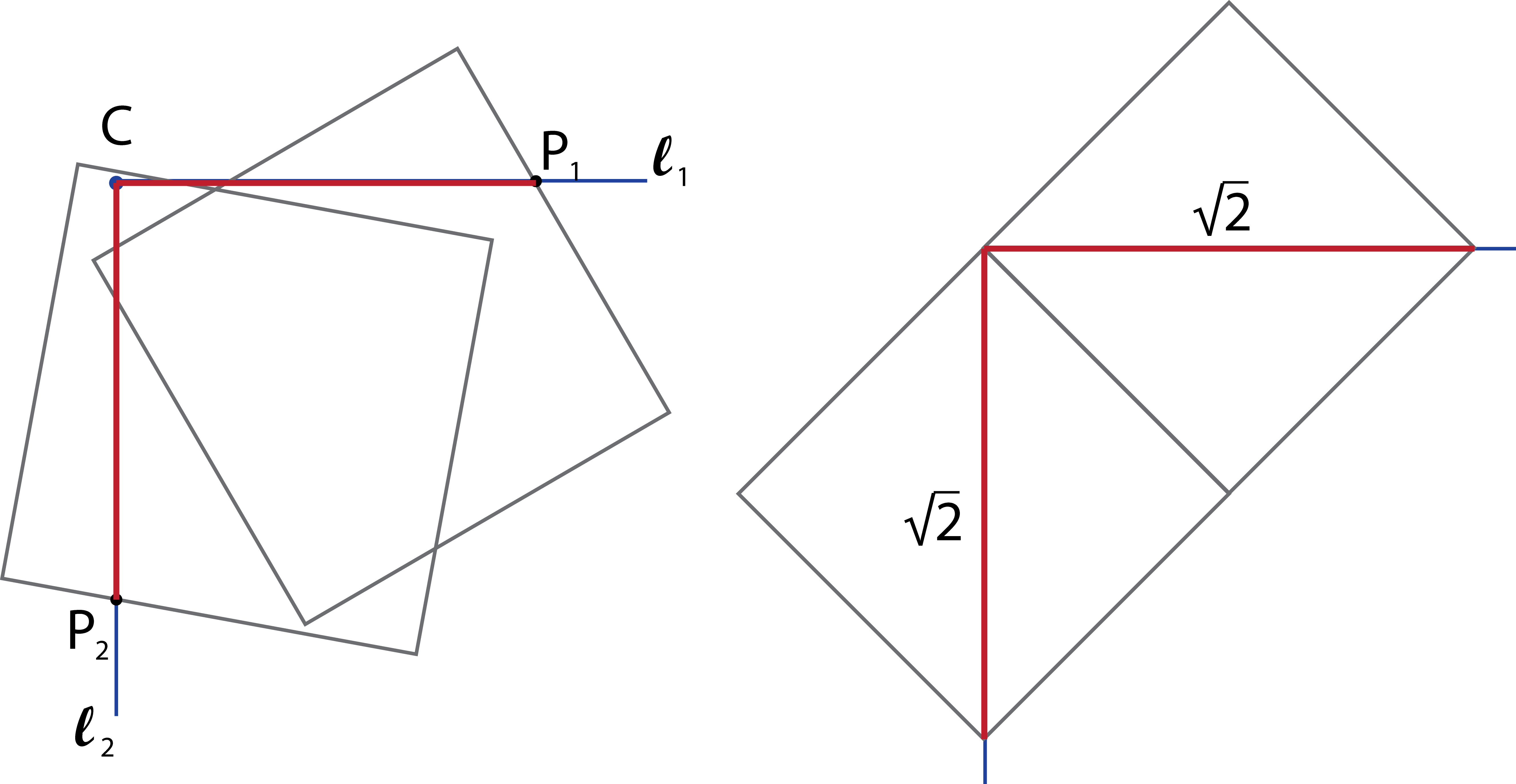}
        \caption{Maximal coverage of $\periS{n}$ by multiple \textit{double-sided perimeter tiles} associated with the same corner of $\faceS{n}$}
        \label{fig:corner-cover}
    \end{figure}
    \end{proof}
\end{lemma}

\section{Spill created by unit squares covering $\periS{n}$}
In this section, we derive a lower bound on the area of a \textit{single-sided perimeter tile} that must spill outside of $\faceS{n}$.

\begin{lemma}\label{lemma:spill-area}
A \textit{single-sided perimeter tile} that intersects with a side of $\periS{n}$ by a segment of length $b > 1$ must spill outside of $\faceS{n}$ by an area of at least $\frac{\sqrt{b^2 - 1}}{2}$.
    \begin{proof}
    Because $b > 1$, the unit square cannot be parallel to the side of $\periS{n}$ and must meet that side at an angle $0 < \theta \leq \frac{\pi}{4}$ (WLOG, we define $\theta$ as the smaller angle). This creates a spill outside of $\faceS{n}$ with an area at least that of a right triangle with hypotenuse of length $b > 1$ and angles $\theta$ and $\frac{\pi}{2}-\theta$ (Figure \ref{fig:spill-area}), which is 
    \begin{equation*}
        \frac{b^2\sin{\theta}\cos{\theta}}{2} = \frac{b^2\sin{2\theta}}{4}
    \end{equation*}

    Also, because the two sides that form the right angle cannot be longer than $1$ (since it was part of a unit square), we have $0 < \sin{\theta} \leq \cos{\theta} \leq \frac{1}{b}$, which imples the following bound for $\theta$
    \begin{equation*}
        \arccos{\frac{1}{b}} \leq \theta \leq \frac{\pi}{4}
    \end{equation*}

    Since $\sin{2\theta}$ is an increasing function on the interval $\left[\arccos{\frac{1}{b}}, \frac{\pi}{4}\right]$, the spill area must be at least
    \begin{equation*}
        \frac{b^2\sin{\left(\arccos{\frac{1}{b}}\right)}\cos{\left(\arccos{\frac{1}{b}}\right)}}{2} = \frac{b^2 \cdot \frac{\sqrt{b^2-1}}{b} \cdot\frac{1}{b}}{2} = \frac{\sqrt{b^2-1}}{2}
    \end{equation*}
    
    \begin{figure}[h]
        \centering
        \includegraphics[width=0.888\linewidth]{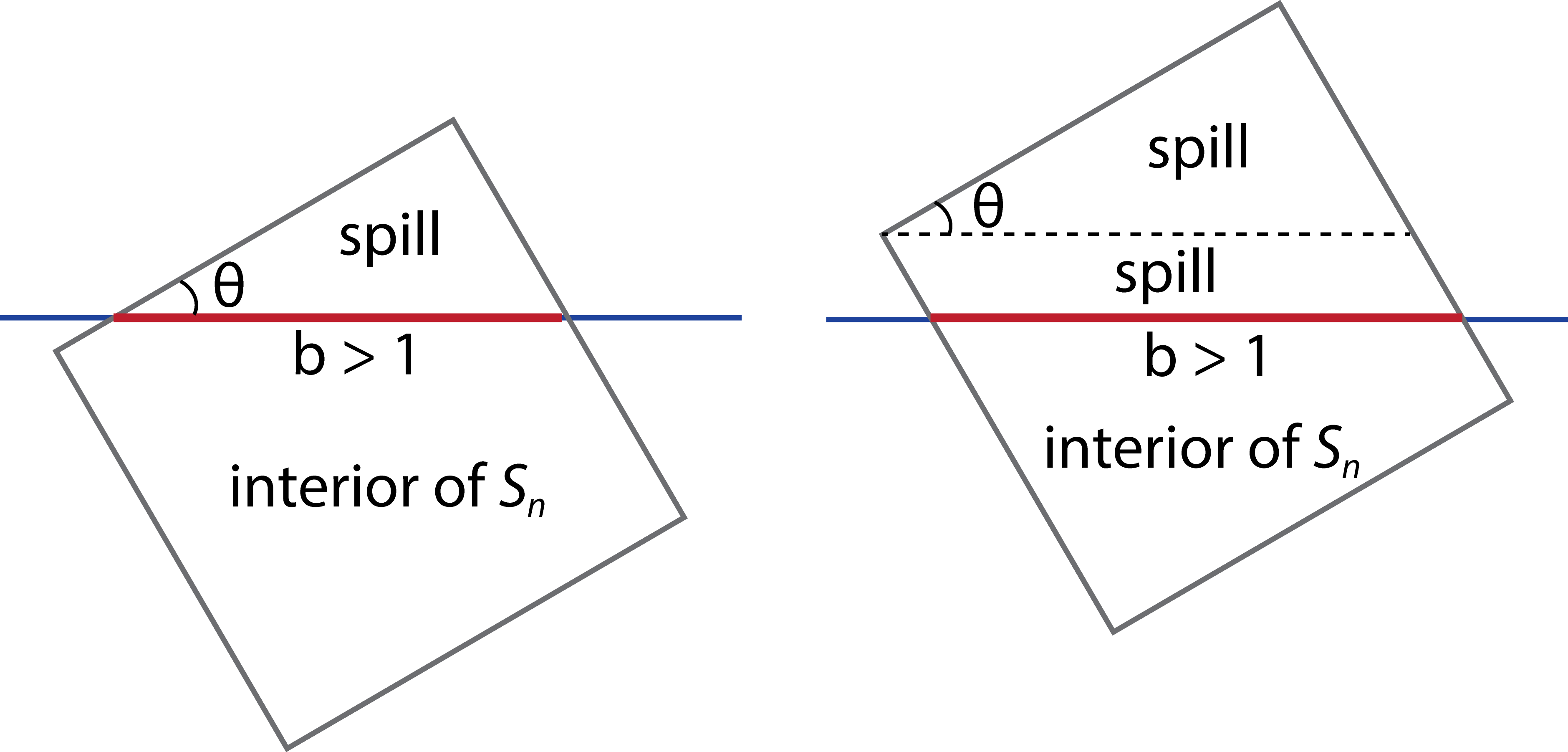}
        \caption{Spill area produced by a \textit{single-sided perimeter tile} which intersects with $\periS{n}$ by $>1$ unit length}
        \label{fig:spill-area}
    \end{figure}
    \end{proof}
\end{lemma}

\section{The proof for $n = 4$}
Let us assume that there is a full covering of $\faceS{4}$ with side length $4 + \epsilon$ by $4^2 + 1 = 17$ unit squares. We will prove that this leads to a contradiction.

By combining the results of the above lemmas, we are now ready to determine the possible values for $n_c$, the number of corners associated with multiple \textit{double-sided perimeter tiles}, and $n_d$ and $n_s$, the number of \textit{double-sided perimeter tiles} and \textit{single-sided perimeter tiles}, respectively.

Our first condition concerns the coverage of $\periS{4}$, which arises from three scenarios: a corner of $\faceS{4}$ covered by one \textit{double-sided perimeter tiles} (there are $4 - n_c$ such corners), a corner of $\faceS{4}$ covered by multiple single \textit{double-sided perimeter tiles} (there are $n_c$ such corners), and \textit{single-sided perimeter tiles} (there are $n_s$ such tiles). Hence, utilizing the fact that the total coverage of $\periS{4}$ achieved by the unit squares cannot exceed the sum of individual maximum, we can derive the following inequality from Lemma \ref{lemma:corner-cover}, \ref{lemma:double-sided-cover}, and \ref{lemma:single-sided-cover-peri}, respectively:
\begin{equation*}
    16 < 4(4 + \epsilon) \leq 2(4-n_c) + 2\sqrt{2}n_c + \sqrt{2}n_s
\end{equation*}
which simplifies to
\begin{equation}\label{eq:peri-bound}
    8 < 2(\sqrt{2} - 1)n_c + \sqrt{2}n_s
\end{equation}

Our second condition concerns the coverage of $\gridS{4}$, which arises from the three tile types: \textit{double-sided perimeter tiles} (there are $n_d$ such tiles), \textit{single-sided perimeter tiles} (there are $n_s$ such tiles), and \textit{interior tiles} (there are $17 - n_d - n_s$ such tiles). Hence, utilizing the fact that the total coverage of $\gridS{4}$ achieved by the unit squares cannot exceed the sum of individual maximum, we can derive the following inequality from Lemma \ref{lemma:double-sided-cover}, \ref{lemma:single-sided-cover-grid}, and \ref{lemma:interior-cover-grid}, respectively.
\begin{equation*}
    40 < 2(4 + 1)(4 + \epsilon) \leq 2n_d + 1.5\sqrt{2}n_s + 2\sqrt{2}(17 - n_d - n_s)
\end{equation*}
which simplifies to
\begin{equation}\label{eq:grid-bound}
    34\sqrt{2} - 40 > 2(\sqrt{2} - 1)n_d + 0.5\sqrt{2}n_s
\end{equation}

Also, from the definition of $n_c$ and $n_d$, we know that $n_d \geq 4 + n_c$ with equality when each of the $n_c$ corners is associated with exactly $2$ unit squares. Adding this to (\ref{eq:grid-bound}) yields
\begin{equation*}
    34\sqrt{2} - 40 > 2(\sqrt{2} - 1)(4 + n_c) + 0.5\sqrt{2}n_s
\end{equation*}
or equivalently,
\begin{equation}\label{eq:grid-bound-2}
    26\sqrt{2} - 32 > 2(\sqrt{2} - 1)n_c + 0.5\sqrt{2}n_s
\end{equation}
Using the fact that $n_c \geq 0$ in (\ref{eq:grid-bound-2}) yields an upper bound for $n_s$
\begin{equation}\label{eq:ns-upper-bound}
    n_s \leq \left\lfloor \frac{26\sqrt{2} - 32}{0.5\sqrt{2}} \right\rfloor = 6 
\end{equation}

Adding (\ref{eq:peri-bound}) to (\ref{eq:grid-bound-2}) gives us the following lower bound for $n_s$
\begin{equation}\label{eq:ns-lower-bound}
    n_s \geq \left\lceil \frac{40 - 26\sqrt{2}}{0.5\sqrt{2}} \right\rceil = 5 
\end{equation}
Hence, we have $n_s \in \set{5, 6}$.

Next, we show that the case $n_s = 5$ is not possible by substituting into (\ref{eq:peri-bound}) and (\ref{eq:grid-bound-2}) to yield a contradiction (also using the fact that $n_c$ is an integer)
\begin{equation*}
    2 = \left\lceil \frac{8 - 5\sqrt{2}}{2(\sqrt{2} - 1)} \right\rceil \leq n_c \leq \left\lfloor \frac{26\sqrt{2} - 32 - 2.5\sqrt{2}}{2(\sqrt{2} - 1)} \right\rfloor = 1
\end{equation*}

Therefore, $n_s = 6$, and (\ref{eq:grid-bound-2}) implies that
\begin{equation*}
    0 \leq n_c \leq  \left\lfloor \frac{26\sqrt{2} - 32 - 3\sqrt{2}}{2(\sqrt{2} - 1)} \right\rfloor = 0
\end{equation*}
which leads to the following key finding
\begin{theorem}\label{thm:ns-count}
A full covering of $\faceS{4}$ using $4^2 + 1=17$ unit squares, if exists, must consist of exactly $4$ \textit{double-sided perimeter tiles} each covering a corner of $\faceS{4}$ and exactly $6$ \textit{single-sided perimeter tiles}. The other $7$ unit squares are \textit{interior tiles} that do not intersect with $\periS{4}$.
\end{theorem}

Next, using Lemma \ref{lemma:spill-area}, we will derive a lower bound for the total spill area produced by the $6$ \textit{single-sided perimeter tiles}. 

First, we note that out of $>16$ unit length of $\periS{4}$, at most $8$ unit length could be covered by the $4$ \textit{double-sided perimeter tiles}. Hence, the remaining $>8$ unit length must be covered by the $6$ \textit{single-sided perimeter tiles}. 

Let $b_1\geq b_2\geq\dots\geq b_6$ denote the lengths of the intersection between these $6$ \textit{single-sided perimeter tiles} with $\periS{4}$, sorted in descending order. We have the follow conditions on $b_i$'s
\begin{equation}\label{eq:b-bound-base}
    \sum_{i=1}^6 b_i > 8 \text{, where } 0 \leq b_i \leq \sqrt{2}
\end{equation}    
If $k$ terms of $b_i$'s are $\leq1$, then the largest $6-k$ terms of $b_i$'s must satisfy
\begin{equation*}
    (6 - k)\sqrt{2} \geq \sum_{i=1}^{6-k} b_i > 8 - k
\end{equation*}
which yields $k \leq \left\lfloor \frac{6\sqrt{2} - 8}{\sqrt{2} - 1} \right\rfloor = 1$. This means that at least $5$ terms of $b_i$'s are $>1$.

Assuming that there are exactly $m$ terms of $b_i$'s that are $>1$, with $m \in \set{5,6}$. We have the following conditions
\begin{align}\label{eq:b-bound}
    b_1&\geq b_2\geq\dots\geq b_m > 1 \nonumber\\
    \sum_{i=1}^m b_i &> 8 - (6 - m) = m + 2
\end{align}
Or equivalently
\begin{align*}
    \sum_{i=1}^m (b_i - 1) &> 2
\end{align*}
It is interesting to note that this condition does not depend on the exact value of $m$, only on the fact that at least one term among $b_i$'s (e.g., $b_1$) is $>1$.

From Lemma \ref{lemma:spill-area}, these $m$ \textit{single-sided perimeter tiles} associated with $b_1,b_2,\dots,b_m$ must produced a total spill area of at least
\begin{equation}\label{eq:spill-area}
    \text{total spill area} \geq \frac{1}{2} \sum_{i=1}^m \sqrt{b_i^2-1}
\end{equation}
Furthermore, because the function $f(x) = \sqrt{x^2 - 1}$ is strictly concave on the interval $(1, \sqrt{2}]$ (its second derivative is $-\frac{1}{(x^2-1)\sqrt{x^2-1}} < 0$ on this interval), $f(x)$ lies on or above the line segment connecting the points $(1, f(1) = 0)$ and $(\sqrt{2}, f(\sqrt{2}) = 1)$. Hence,
\begin{equation*}
    \sqrt{x^2 - 1} = f(x) \geq \frac{x - 1}{\sqrt{2} - 1} \text{, for } x \in (1, \sqrt{2}]
\end{equation*}
Using this fact in (\ref{eq:spill-area}) yields
\begin{align*}
    \text{total spill area} &\geq \frac{1}{2} \sum_{i=1}^m \sqrt{b_i^2-1} \geq \frac{1}{2} \sum_{i=1}^m \frac{b_i-1}{\sqrt{2} - 1} \\
    &= \frac{1}{2(\sqrt{2} - 1)} \sum_{i=1}^m (b_i-1) \\
    &> \frac{2}{2(\sqrt{2} - 1)} \approx 2.414
\end{align*}

This is a contradiction because if the total spill area is $>1$, then these $17$ unit squares would have to fully cover $\faceS{4}$, whose area is $>16$, using the remaining area of $<16$.

Therefore, there is no full covering of $\faceS{4}$ using only $17$ unit squares.

\section{Implication for related problems}
Our idea to simultaneously consider the coverages of $\periS{n}$ and $\gridS{n}$ is incredibly tight for the case of $n = 4$ and $17$ unit squares, both for determining the exact values of $n_s$ and $n_c$ and for deriving the contradiction via the spill area produced by $n_s$ \textit{single-sided perimeter tiles}.

Increasing the number of unit squares to $\geq18$ affects only the $26\sqrt{2}$ terms in (\ref{eq:ns-upper-bound}) and (\ref{eq:ns-lower-bound}) and yield much looser bounds for $n_s$
\begin{align*}
    \text{With $18$ unit squares; } 1 = \left\lceil \frac{40 - 28\sqrt{2}}{0.5\sqrt{2}} \right\rceil \leq n_s &\leq \left\lfloor \frac{28\sqrt{2} - 32}{0.5\sqrt{2}} \right\rfloor = 10 \\
    \text{With $19$ unit squares; } 0 = \left\lceil \frac{40 - 30\sqrt{2}}{0.5\sqrt{2}} \right\rceil \leq n_s &\leq \left\lfloor \frac{30\sqrt{2} - 32}{0.5\sqrt{2}} \right\rfloor = 14
\end{align*}
Using the fact that $n_c \leq 4$ in (\ref{eq:peri-bound}) yields a slightly better lower bound for $n_s$
\begin{equation*}
    n_s \geq \left\lceil \frac{16 - 8\sqrt{2}}{\sqrt{2}} \right\rceil = 4
\end{equation*}

None of these possible values for $n_s$ can be ruled out by plugging into (\ref{eq:peri-bound}) and (\ref{eq:grid-bound}) to derive a contradiction through $n_c$ as before. Furthermore, almost all of these cases permit $n_c > 0$ which considerably weakens the condition (\ref{eq:b-bound-base}) on the spill area produced by $n_s$ \textit{single-sided perimeter tiles} by reducing the fractional length of $\periS{4}$ that must be covered by \textit{single-sided perimeter tiles}. This fact in combination with an increased number of \textit{single-sided perimeter tiles} means that none of the $b_i$'s defined in (\ref{eq:b-bound-base}) has to be $>1$.

Similarly, applying our technique to the covering of a square of side length $5 + \epsilon$ with $5^2 + 1 = 26$ unit squares yields the following loose bounds for $n_s$
\begin{equation*}
    7 = \left\lceil \frac{20 - 8\sqrt{2}}{\sqrt{2}} \right\rceil \leq n_s \leq \left\lfloor \frac{44\sqrt{2} - 52}{0.5\sqrt{2}} \right\rfloor = 14
\end{equation*}

For the case $n = 4$ with $\geq 18$ unit squares, it should be possible to improve the bound on the spill area by including those produced at the $n_c$ corner(s) associated with multiple \textit{double-sided perimeter tiles} and also considering the overlap among multiple \textit{double-sided perimeter tiles} at the same corner.

For $n \geq 5$, we can foresee how our technique would fail, since the coverages of $\periS{n}$ and $\gridS{n}$ scale with $n$ while the number of allowed unit squares scales with $n^2$.

\section{Acknowledgments}
The author would like to thank Prof. Wacharin Wichiramala and Kirati Sriamorn from the Department of Mathematics and Computer Science, Chulalongkorn University and Prof. Supanut Chaidee from the Department of Mathematics, Chiang Mai University for introducing me to this problem and for their discussion that contributed to various ideas in this work. This research was carried out during the 2026 geometry and combinatorics research boot camp organized at Chiang Mai University, Thailand, on July 26-30, 2026.

\section{Disclosure statement}
The author has no competing or financial interest to declare.

\subsection{Declaration of generative AI use}
Claude (a mix of Fable, Opus, and Sonnet version 5) was used during the research to develop numerical optimization routines to explore the bounds and check some conjectures. However, all parts of the final analytical proof were derived without the aid of AI.

\bibliographystyle{amsplain}
\bibliography{reference}

\end{document}